\documentclass[10pt]{article}
\usepackage{mathpazo, courier, pifont}
\usepackage[top=3cm,left=2.5cm,right=3cm,bottom=2cm]{geometry}
\usepackage[T1]{fontenc}
\usepackage{textcomp}
\usepackage{color}
\usepackage[latin1]{inputenc}
\usepackage{amsmath,amsthm,amsfonts,amssymb,amscd}
\usepackage[mathcal]{eucal}
\usepackage{psfrag}
\usepackage{ifthen}

\ifx\pdftexversion\undefined
  \usepackage[dvips]{graphicx}
  \DeclareGraphicsExtensions{.eps}
\else
  \usepackage[pdftex]{graphicx}
  \DeclareGraphicsExtensions{.png}
\fi

\ifthenelse{\isundefined{\notusefancy}}
{
  \usepackage{fancyhdr}
  \usepackage{datetime}
  \ddmmyyyydate
  \pagestyle{fancy}
}

\theoremstyle{plain}
\newtheorem{lemma}{Lemma}[section]
\newtheorem{corollary}[lemma]{Corollary}

\newtheorem{maintheorem}{Theorem}

\theoremstyle{definition}

\newtheorem{remark}[lemma]{Remark}
\newtheorem{definition}[lemma]{Definition}

\newcommand{\NN}{\mathbb{N}}

\DeclareMathOperator{\supp}{supp} 
\DeclareMathOperator{\dist}{dist}

\title{Every expanding measure has the nonuniform specification property}
\author{Krerley Oliveira
\thanks{Work partially supported by CNPq, CAPES, FAPEAL,
INCTMAT and PRONEX.}}
\date{ }

\begin{document}

\maketitle

\begin{abstract}

Exploring abundance and non lacunarity of hyperbolic times for
endomorphisms preserving an ergodic probability with positive Lyapunov
exponents, we obtain that there are periodic points of period
growing subexponentially with respect to the lenght of almost
every dynamical ball. In particular, we conclude that any ergodic
measure with positive Lyapunov exponents satisfy the nonuniform
specification property.

As consequences, we (re)-obtain estimates on the recurrence to a
ball in terms of the Lyapunov exponents and we prove that any
expanding measure is limit of Dirac measures on periodic points.
\end{abstract}

\section{Introduction}
A basic problem in Dynamical Systems and its Applications is to study the existence and abundance of periodic points and  understand its distribution on the underlying space. In simple terms, for a differentiable map $f$ in a Riemman manifold $M$, ones want to analyse which conditions on $f$ and its derivatives ensure the existence of periodic points and how they are distributed in $M$.  

 For uniformly hyperbolic maps, Bowen stablished in \cite{Bo72} that the asymptotical exponential growth of the set $P_n(f)$ of periodic points of period $n$ was
determined by the topological entropy:
$$
\lim\limits_{n\rightarrow \infty} \frac{\log P_n(f)}{n}=h_{top}(f). 
$$   
He introduced the important  notion of \emph{specification} by periodic orbits and  proved a number of important results concerning the
uniqueness and the ergodic properties of equilibrium states, asymptotic growth and the limit distribution of periodic orbits and so on, for Axiom A diffeomorphisms
and flows (\cite{Bo72}).  In words, a system has the specification property if  a (small) error is fixed, given a piece of orbit of size $n$ there exists a periodic points that follows (up to this error) this orbit up to the moment $n$ and has period $n+K$, where $K$ depends only on the error and do not depend on $n$. 

More precisely, we say that
$f$ has the {\em specification property} if there exists
$\epsilon_0>0$ such that for all $x \in M$ and $0 <
\epsilon < \epsilon_0$, $n \leq 1$, there exists a periodic point $p\in M$
such
\begin{itemize}
\item $d(f^i(p), f^i (x))<\epsilon$, for $i=0,\dots,n$
\item $p$ has period less than $n + K$, with $K=K(\epsilon)$
depending only on $\epsilon$.
\end{itemize}

For convenience, we define the \emph{dynamical ball} of size $\epsilon$ and lenght $n$ by
\begin{equation}\label{eq:01}
B_n(x,\epsilon)= \bigcap_{k=0}^{n} f^{-k}(B(f^k(x),\epsilon)
\end{equation} and this can be rephrased just saying that there exists a periodic point of period $n+K(\epsilon)$ in $B_n(x, \epsilon).$
\newpage

  Beyond the uniformly hyperbolic setting, the understanding of periodic orbits and its structure is much less developed.  A remarkable work stablishing connections between Lyapunov exponents of a given measure and  periodic points  and its distribution was obtained by Katok in \cite{Ka80}. There, he proved a technical lemma known as \emph{Katok Closing Lemma}: roughly, it tell us that  if $f$ is $C^{1+\alpha}$ diffeomorphism and $x$ is a recurrent point in a Pesin's block such that $f^n(x)$ is in the same Pesin's block, the orbit of $x$ up to the moment $n$ can be shadowed by a periodic point $p$ with period depending only on the shadowing constant and the Pesin block, not on $n$ (see \cite{Ka80}, Section 3 for precise statements). Recent improvements of this result include the papers \cite{SuTi10, LLS09}.
  
  Here, we improve part of this results relaxing the hypothesis in \cite{Ka80} and  obtaining a quatitative version of Katok's Closing Lemma for non-uniformly expanding maps preserving a ergodic measure with positive exponents.   We study the extension of the Bowen's specification property in a measure-theoretical setting.  
  
  Consider an ergodic invariant measure $\mu$ with only positive Lyapunov exponents for any $C^1$ endomorphisms $f$ with non-flat critical set. We are able to show that for $\mu$ almost
every point, given a natural number $n$ and $\epsilon>0$ there
exists a periodic point $p$ on the dynamical ball
$B_n(x,\epsilon)$ with period $K(n,\epsilon)$ that
growth assymptotic like $n$ at infinity.

This generalize the Katok´s  Closing Lemma  in
several ways. First, we are able to deal with $C^1$ maps, instead
$C^{1+\alpha}$ maps, since we do not need to make use of Pesin´s
theory. Moreover, we are able not only to prove the existence of a
shadowing point but to obtain quantitative estimates on its
period. As consequence of our result, we are able to obtain estimates on Poincaré recurrence in terms of the Lyapunov exponents. Let us describe it in detail.   
    
The study of recurrence and return times are among the most
prolific tools for better understanding of statistical
properties of dynamical system. The most basic concept in this context is the {\em Poincaré recurrence of a set}.
Given a measurable dynamical system $(M,\mu, f)$ and a measurable
set $A \subset M$,
we define its Poincaré recurrence as
\[
\tau(A) := \inf \{ n \in \NN; f^n(A) \cap A \neq  \emptyset\}.
\]

In the literature, many relations have been established between recurrence times
and other important aspects of dynamical systems such as
entropy, Hausdorff dimension, mixing properties  and Lyapunov exponents.

In order to grasp a finer understanding of these relations it is
useful consider return times associated to shrinking neighborhoods
such as decreasing sequences of balls or cylinders. 

Closely related to the Poincaré recurrence of $B_n(x,\epsilon)$ we may
ask for the existence of periodic points in this dynamical ball.
Sometimes it is possible to find a periodic point $z$ in a given
ball but often its period is unrelated to $n$ or $\epsilon$. In
the scenario we are facing to, we can restrict our attention to
investigate how frequently dynamical balls contain periodic points
of (at some extent) controlled period. This gives place to a
notion of specification that  bounces back to Bowen itself \cite{Bo72, Bo74}. 

In \cite{Sa03} it is presented the nonuniform counterpart of these
notions allowing that the radius $\epsilon$ in Equation~(\ref{eq:01})
decrease with $n$. To begin with, let $q:M \to [0,\infty)$ be a
$\eta$-slowly varying function, that is to say, a function
satisfying \;$q(f(x)) \leq e^\eta q(x)$, for all $x \in M$ and
some fixed $\eta > 0$. A $(n,\epsilon,q)$ {\em nonuniform
dynamical ball} is defined as
\[
\widetilde{B}_n(x,\epsilon)= \bigcap_{k=0}^{n} f^{-k}(B(f^k(x),\epsilon q(f^k(x))^{-2})).
\]
And we say that $(f,\mu)$ has the {\em nonuniform specification
property} if for $\mu$-almost everywhere, given a $\eta$-slowing
varying function $q$,  the ball $\widetilde{B}_n(x,\epsilon)$
contains a periodic point whose period is less than $n +
K(n,\epsilon,\eta)$ satisfying
\[
 \lim_{\eta \to 0}\,\limsup_{n \to \infty} \, \frac{K(n,\epsilon,\eta)}{n} = 0.
\]

In the same work Saussol et al \cite{Sa03} proved that the
nonuniform specification property implies an estimate of the
recurrence time for arbitrary positive $\mu$-measure sets in terms
of Lyapunov exponents.

Here we are able to obtain a more natural result showing that
positiveness of all Lyapunov exponents implies the nonuniform
specification property. The main idea that we use to obtain this result,  is the notion of
hyperbolic time. This notion has been used by many authors to
obtain statistical properties of dynamical systems, such as
existence and uniqueness of SRB measures (\cite{ABV00}), stochastic stability (\cite{AA03}), infinite Markov Partitions (\cite{Pi09}) and equilibrium states
(\cite{Ol03, OV08}). Along this paper, we prove that almost
every point with respect to any ergodic measure with positive
exponents has a nonlacunary sequence of hyperbolic times. Let us state our  main result:

\begin{maintheorem}\label{th:t1}
Let $f: M \to M$ be a  $C^1$ map with non-flat
critical set $\mathcal{C}$. Any $f$-invariant expanding
measure $\mu$ satisfies the {\em nonuniform specification
property}.
\end{maintheorem}

See Section~\ref{preliminaries} for precise definitions. Using Theorem~\ref{th:t1} above and the Theorem of \cite{Sa03}, we
obtain that:

\begin{corollary}\label{c:main}
Let $f: M \to M$ be a  $C^1$ map with non-flat
critical set $\mathcal{C}$, preserving an $f$-invariant expanding
measure $\mu$. For $\mu$ almost every $x\in M$
\[
 \limsup_{r \to 0} \frac{ \tau(B(x,r)) }{ - \log r  } \leq \frac{1}{\lambda_1}.
\]

Moreover, if we assume that $h_\mu(f)>0$ the following inequality also holds true
$$
\frac{1}{\lambda_n} \leq \liminf_{r \to 0} \frac{ \tau(B(x,r)) }{ - \log r  }.
$$
\end{corollary}

It is interesting to observe that the inequalities presented in
Corollary \ref{c:main} may be attained, but there are examples
where these inequalities are strict. In fact, there exists a
linear expanding map on the two dimensional torus preserving the
Lebesgue measure and with Lyapunov exponents $0<\lambda_1
<\lambda_2$ such that
$$
\frac{1}{\lambda_2}<\frac{2}{\lambda_1+\lambda_2}=\lim_{r \to 0} \frac{ \tau(B(x,r)) }{ - \log r  } <\frac{1}{\lambda_1},
$$ for Lebesgue almost every point $x \in M$.
See Section 4 of \cite{Sa03}, for further details and some other examples.

Another interesting consequence of Theorem~\ref{th:t1} above, is
the fact that every expanding measure is approximated in the
weak$^{\star}$ topology by Dirac measures at periodic points.

\begin{corollary} Let $f: M \to M$ be a  $C^1$ map with non-flat critical set
$\mathcal{C}$, preserving an $f$-invariant expanding measure $\mu$.
Then,  $\mu$ is
the weak$^{\star}$ limit of Dirac measures at periodic points.
\end{corollary}

We point out that in \cite{Va10} the author make use of the notion of nonuniformly specification to establish bounds for the measure of deviation sets associated
to continuous observables with respect to not necessarily invariant weak Gibbs
measures. Under some mild assumptions, he obtained upper and lower bounds
for the measure of deviation sets of some nonuniformly expanding maps, in-
cluding quadratic maps and robust multidimensional nonuniformly expanding
local diffeomorphisms. To describe more precisely the notions in this introduction, we recall
some basic notions in Ergodic Theory in the next section.

\section{Preliminaries}\label{preliminaries}

By a classical theorem due to Oseledets (\cite{Os68}), given a $f$-invariant
measure $\mu$, for almost every point  $x \in M$ there exists an $f$-invariant (measurable)
splitting
$$
E_1(x)\oplus \dots \oplus E_{k(x)}(x)=T_x M
$$
and numbers $\overline{\lambda}_1(x)< \overline{\lambda}_2(x)< \dots
<\overline{\lambda}_{k(x)}(x)$ such that given $v \in E_i(x)\setminus\{0\}$, we have that
$$
\overline{\lambda}_i(x)=\lim_{n\rightarrow \infty} \frac{1}{n} \log
\|Df^n(x)v\|.
$$ Note that the last expression implies that each $\overline{\lambda}_i$ is an $f$-invariant function, i.e.,
$\overline{\lambda}_i(x)=\overline{\lambda}_i(f(x))$. In
particular, if $\mu$ is ergodic the functions
$\overline{\lambda}_i$ are constant almost everywhere. Let
$$
\lambda_1(x) \leq \lambda_2(x) \leq \dots\leq \lambda_d(x)
$$ be
the numbers $\overline{\lambda}_j(x)$, in a non-decreasing order,
and each repeated with multiplicity
 $\text{dim} (E_i(x))$. These numbers are called \emph{the Lyapunov exponents} of $f$ at the point $x$.

\begin{definition} A map $f: M \to M$  is {\em strongly
transitive} on a set $X$, if for any open set  $U \subset M$ such
that $U \cap X \neq \emptyset$,  there exists $n \in \mathbb{N}$
such that $X \subset \bigcup_{j=0}^n f^j(U)$.
\end{definition}

We recall that the \emph{support} of a invariant measure $\mu$ is the full measure set $\supp(\mu)$ of all points such any neighbourhood has positive measure. 
 
\begin{definition} We say that an ergodic $f$-invariant probability is  an
 \emph{expanding measure}
if all its Lyapunov exponents are positive and $f$ is strongly transitive on support of $\mu$.
\end{definition}

%
%

Let us present now a family of maps introduced in \cite{ABV00}.

\begin{definition} Given a $C^1$ map $f\colon M\to M$, we say that
$\mathcal{C}\subset M$ is  a $\beta$-{\em non-degenerate critical
set }\index{non-degenerate!critical set} if there exists $B>0$ such
that the following two conditions hold.
\begin{enumerate}
\item[(1)]
\quad $\displaystyle{\frac{1}{B}\dist(x,\mathcal{C})^{\beta}\leq
\frac {\|Df(x)v\|}{\|v\|}\leq B\,\dist(x,\mathcal{C})^{-\beta} },$
for all $v\in T_x M$.
\end{enumerate}
For every $x,y\in M\setminus\mathcal{C}$ with
$\dist(x,y)< \dist(x,\mathcal{C})/2$ we have
\begin{enumerate}
\item[(2)] \quad $\displaystyle{\left|\log\|Df(x)^{-1}\|-
\log\|Df(y)^{-1}\|\:\right|\leq
\frac{B}{\dist(x,\mathcal{C})^{\beta}}\dist(x,y)}$.

\item[(3)] \quad There exists $K>0$ such that $\|\det Df(x)\|<K\dist(x,\mathcal{C})^\beta$ for $x \in M\setminus \mathcal{C}$.

\item[(4)] \quad There exists $C>0$ such that $\|Df(x)\|<C$ for $x \in M\setminus \mathcal{C}$.

\end{enumerate}

\end{definition}

A map $f\colon M\to M$  will be called {\em non-flat} if $f$ is
local $C^1$ diffeomorphism in the whole manifold  except  in  a
$\beta$-{\em non-degenerate critical set }  $\mathcal{C}\subset
M$, for some $\beta$.

%
%
%


\section{Positive Exponents and Hyperbolic Times}

Throughout we assume that  $f: M \to M$ be a $C^1$
map with non-flat critical set $\mathcal{C}$ which preserves
an expanding ergodic invariant measure $\mu$. We also assume that $f$ is strongly transitive on the support of $\mu$. It is well-known (see \cite{Pe77}) that expanding measures admit

invariant unstable local manifolds at almost every point.
However, the dependence of these manifolds with the base
point is just measurable and several of its features are not suitable for computations.
In particular, the size of this local manifold is just a measurable
function and this is an additional challenge when we need to handle
these objects.
In the next definition, we introduce a concept that address some of these difficulties:

\begin{definition} \label{def:hiperbolic_times}
Given $c>0$ and $\delta>0$, we say that $n$ is a
\emph{$(c, \delta)$-hyperbolic time} for a point $x \in M$, if
we can find a neighborhood $V_n(x)$ of $x$
such that
$f^n : V_n(x) \rightarrow B_\delta(f^n(x))$ is a homeomorphism with the property that given $x_1, x_2 \in V_n(x)$, then
$$
\dist(f^k(x_1), f^k(x_2))\leq e^{-2c(n-k)}\dist(f^n(x_1), f^n(x_2)),
$$ for all $0\leq k\leq n$.
\end{definition}

In this context, we call $V_n(x)=f^{-n}\big (
B_\delta(f^n(x))\big)$ the {\em hyperbolic pre-ball} of length
$n$ and radius $\delta$ at $x$.

\begin{definition} We consider $H_n(c,\delta,f)$ the set of points $x$
such that $n$ is a hyperbolic time of $x$ and define
\[
H(c,\delta,f):=
\bigcap\limits_{m\geq 1} \bigcup\limits_{n\geq m} H_n(c,\delta,f)
\]
the set of points with infinitely many $(c,\delta)$-hyperbolic times for $f$.
\end{definition}

\begin{definition} Given $0<\theta<1$, let  $H^\theta(c,\delta,f)$ be the set of points with \emph{frequency} of $(c,\delta)$-hyperbolic times at least $\theta$:
\[
H^\theta(c,\delta,f) := \{x\in H(c,\delta,f);
\limsup\limits_{n\rightarrow \infty} \frac{\#\{k;\; 1\leq k\leq n,
x \in H_k(c,\delta,f)\}}{n}\geq \theta\}.
\]

\end{definition}

From \cite{ABV00}, we have a sufficient criterium for
abundance of $(c,\delta)$-hyper\-bolic times at a point $x$.
We say that $f$ is \emph{asymptotically  $c$-expanding} at a point $x$, if
\begin{equation}\label{cond.exp}
\limsup\limits_{n\rightarrow \infty} \frac{1}{n} \sum\limits_{i=0}^{n-1} \log\|Df(f^{i}(x))^{-1}\|^{-1}>4c.
\end{equation}

Furthermore, we say that the point $x$ satisfies the condition of
\emph{slow approximation} to the critical set if for all
$\epsilon>0$ there exists $\delta>0$ such that
\begin{equation}\label{cond.sa}
\limsup\limits_{n\rightarrow \infty} \frac{1}{n} \sum\limits_{i=0}^{n-1} -\log \text{dist}_\delta(f^i(x), \mathcal{C}) \leq \epsilon.
\end{equation}

Here, $\dist_\delta$ is the $\delta$-{\em truncated} distance:
\[
\dist_\delta(x, \mathcal{C}) := \dist(x, \mathcal{C}),\;\;\text{
whenever } \dist(x, \mathcal{C}) < \delta \text{ and } 1 \text{
otherwise }.
\]

\begin{lemma}\label{l.pliss}
There exists $\delta>0$ and $\theta>0$ depending only on $c$ and $f$,  such that if $x$ is a point satisfying Equations~(\ref{cond.exp}) and~(\ref{cond.sa}), then $x$ belongs to $H^\theta(c,\delta,f)$.
\end{lemma}
\begin{proof}
The result follows from Pliss' Lemma \cite{Pl71}. For details see \cite[Corollary 3.2]{ABV00}.
\end{proof}

%
%

Abundance of hyperbolic times is given by

\begin{lemma}\label{l.hyptimes} Given $\mu$ an ergodic expanding measure, there exists $\ell \in \mathbb{N}$, and real numbers $c, \theta, \delta>0$ such that
$$
\mu\big(H^\theta(c,\delta,f^{\ell})\big)=1.
$$
\end{lemma}

\begin{proof}
We observe that since the Lyapunov exponent of $\mu$ is positive,
we may find a constant $c>0$ such that for almost every $x \in M$,
there exists $n_0(x) \in \mathbb{N}$ such that if $n \geq n_0(x)$
then $\|Df^n(x)v\| \geq e^{8cn}\|v\|$, for every $v \in T_x M$.
This is equivalent to $\|Df^n(x)^{-1}\|\leq e^{-8cn}$. Denote by
$A_k$ the set
$$
A_k=\{x \in M; n_0(x)\leq k\}.
$$
It is clear that $\mu(A_k^c)$ goes to zero when $k$ goes to
infinity. Observe that
\begin{equation}\label{eq.A}
\int_M \frac{1}{k} \log \|Df^k(x)^{-1}\|\,d\mu = \int_{A_k} \log
\frac{1}{k} \|Df^k(x)^{-1}\|\,d\mu + \int_{A_k^c} \log \frac{1}{k}
\|Df^k(x)^{-1}\|\,d\mu  \leq $$
$$ -8c\mu(A_k)+ \int_{A_k^c}
\frac{1}{k} \log \|Df^k(x)^{-1}\|\,d\mu.
\end{equation}

On the other hand,

$$
\int_{A_k^c} \frac{1}{k}
\log \|Df^k(x)^{-1}\|\,d\mu \leq \int_{A_k^c} \frac{1}{k}
\sum\limits_{i=0}^{k-1}\log \|Df(f^i(x))^{-1}\|\,d\mu.
$$

 Since $\log\|Df(x)^{-1}\|$ is integrable, by Birkhoff´s Ergodic Theorem, the function $h_k (x)= \frac{1}{k}
\sum\limits_{i=0}^{k-1}\log \|Df(f^i(x))^{-1}\|$ converges in
$L^1(d\mu)$ to some function $\varphi$. Using that  $\mu(A_k^c)$ goes
to zero when $k$ goes to infinity, we have that
$$
\lim_{k\rightarrow \infty} \int_{A_k^c} \frac{1}{k}\log
\|Df^k(x)^{-1}\|\,d\mu =0.
$$
Observing this and Equation~(\ref{eq.A}) above, there exists $l_0$ such
that if $\ell \geq  l_0$
$$
\int_M \frac{1}{\ell} \log \|Df^\ell(x)^{-1}\|\,d\mu < -4c<0.
$$
In particular, using Birkhoff´s Ergodic Theorem once more, the
Condition~(\ref{cond.exp}) holds for the function $f^\ell$ at almost every point $x \in M$. Since $\mu$ is an ergodic measure and  has  finite Lyapunov exponents with respect to $\mu$, we have by Lemma 10.2 of \cite{Pi09} that Condition~(\ref{cond.sa}) holds for $f^\ell$ at almost every point. Since Condition~\ref{cond.exp} and~\ref{cond.sa} are satisfied, we have by Lemma~\ref{l.pliss} that there exists $\theta>0$ such that almost every $x \in M$ belongs to $H^\theta(c,\delta,f^{\ell})$.

\end{proof}

We fix $\ell, c, \delta, \theta>0$ as in Lemma above,  and
consider $g=f^{\ell}$.  When we say just \emph{hyperbolic time},
we mean $(c,\delta)$-hyperbolic time with respect to $g$. Define the \emph{first hyperbolic time}  function $n_1:
H^\theta(c,\delta,g) \rightarrow \mathbb{N}$ setting  $n_1(x)$ as
the first hyperbolic time of $x$. 

\begin{remark}\label{r.1}
Observe that if $m$ is a
hyperbolic time for $x$ and $n$ is a hyperbolic
 time for $f^m(x)$,
then $n+m$ is a hyperbolic time for $x$. From this follows  that if
$n_1(x)<n_2(x)<\dots$
denotes the sequence of hyperbolic times of $x$, then
$$n_1 (f^{n_i(x)}(x))= n_{i+1}(x)-n_i(x).
$$
\end{remark}

For the next result we make use of \cite[Lemma 4.7]{Pi09}:

\begin{lemma}\cite[Lemma 4.7]{Pi09}\label{l.Pinheiro} Let $(G_j)_{j\in \mathbb{N}}$ be a collection of subsets of M such that  for all 
$x\in G_n$ and $0\leq j<n$ we have that $g^j(x)\in G_{n-j}$. Let $B\subset M$ and let $x \in B$ be a point such that 
$$
\#\{j\geq 1; x\in G_j\text{ and }g^j(x)\in B\}=+\infty.
$$ Consider $\mathcal{O}^+(x)$ the positive orbit of $x$ and let $T:\mathcal{O}^+ \cap B \rightarrow \mathcal{O}^+ \cap B$ the function defined by $T(y)=g^{\varphi(y)} (y)$, where $\varphi(y)=\min\{j\in \mathbb{N}; y \in G_j \text{ and } g^j(y)\in B\}$. Then, if 
$$ \limsup_n \frac{1}{n} \#\{1\leq j\leq n;x \in G_j \text{ and } g^j(x) \in B\} >\theta>0
$$ we have that
$$
\liminf\limits_{n\rightarrow \infty} \frac{1}{n} \sum\limits_{j=0}^{n-1} \varphi(T^j(x))\leq \theta^{-1}.
$$
\end{lemma}

Using  this lemma we are able to prove that

\begin{lemma}\label{lem:01}
 The first hyperbolic time function of $g$  is integrable:
\[
 \int n_1(x) d\mu < +\infty.
\]
\end{lemma}
\begin{proof}
Observe that by Lemma~\ref{l.hyptimes}, the set $H^\theta(c,\delta,g)$ has full measure. Since $\mu$ is ergodic, by Birkhoff's Ergodic Theorem, to show the integrability of $n_1$ is enough to verify that  for every $x\in H^\theta(c,\delta,g)$ we have 
$$
\lim\limits_{n\rightarrow \infty} \frac{1}{n} \sum\limits_{j=0}^{n-1} n_1(g^{n_j(x)}(x))\leq \theta^{-1},
$$ where we set $n_0(x)=0$. Indeed, taking $G_j=H_j(c,\delta,g)$ and $B=H^\theta(c,\delta,g)$ as in Lemma~\ref{l.Pinheiro} above,  it follows from  Remark~\ref{r.1} that if $x\in G_n$ then $g^j(x)\in G_{n-j}$  for every $0\leq j<n$. Observe that by the definition, for every  $x\in H^{\theta}(c,\delta,g)$ we have that
$$ \limsup_n \frac{1}{n} \#\{1\leq j\leq n;x \in G_j \text{ and } g^j(x) \in B\} >\theta>0.
$$ At last, we have that $\varphi(x)=\min\{j\in \mathbb{N}; y \in H_j(c,\delta,g) \text{ and } g^j(x)\in B\}=n_1(x)$. Thus, the inequation above follows straight forwardly from Lemma~\ref{l.Pinheiro}.       

\end{proof}

An increasing sequence $(a_k)_{k\in \mathbb{N}}$ of natural numbers
is called \emph{nonlacunary}, if
$$
\lim\limits_{k\rightarrow \infty} \frac{a_{k+1}}{a_k} =1.
$$
In \cite{OV08}, this notion is used in the context of equilibrium states to prove existence and uniqueness of a special type of weak Gibbs measure, called
\emph{nonlacunary} Gibbs measure. Therein, they proved that the integrability of the first hyperbolic time implies
nonlacunarity. Here, we slightly generalize this result. Let $\gamma: \mathbb{R}^+ \rightarrow \mathbb{R}^+$ a bijection. We say that a increasing sequence $(a_k)_{k \in \mathbb{N}}$ is $\gamma$-nonlacunary, if 

$$
\lim\limits_{k\rightarrow \infty} \frac{a_{k+1}-a_k}{\gamma(a_k)} =0.
$$      

In particular, if $\gamma$ is the identity, a $\gamma$-nonlacunary sequence is just a nonlacunary sequence. 

\begin{lemma}\label{l.nonlacunargeneral}
If $\gamma^{-1}$ denotes the inverse function of $\gamma$, assume that for every $r>0$ the function  $\gamma^{-1}\circ(r n_1) $ is $\mu$-integrable. Then, for $\mu$-a.e. $x \in M$ the sequence $n_j(\cdot)$ of its hyperbolic times is $\gamma$-nonlacunary.
\end{lemma}

Before prove this Lemma, we remark that  if we know \emph{a priori} that $\mu \big(H_n(c,\delta,f)\big)$ decays in particular way, then we may take $\gamma(t)$ growing less than $t$ at the infinity. For example, $\mu \big(H_n(c,\delta,f)\big)$ decays expoentially,  the hypothesis of the lemma above is easily satisfied for any function of the form $\gamma(t)=t^p$, where $p>0$. 

\begin{proof}

Let $D$ be the set of points for which the sequence $n_j(\cdot)$
fails to be $\gamma$-nonlacunary. For each $r>0$, define $L_r(n)
= \{x\in M: n_1(x)\geq r \gamma(n) \}$. If $x\in D$ then there exists
a rational number $r>0$, and there are infinitely many values
of $i$ such that $n_{i+1}(x)-n_i(x)\ge r\gamma(n_i(x))$. By
Remark~\ref{r.1}, the latter implies that
$$
n_1(f^{n_i}(x))= n_{i+1}(x)-n_i(x) \ge r \gamma(n_i(x)).
$$
So, there are arbitrarily large values of $n$ such that $x\in
f^{-n}(L_r(n))$. In other words, $D$ is contained in the set
$$
L= \bigcup_{r\in \mathbb{Q}\cap (0,+\infty)}
   \bigcap_{k=0}^{\infty} \bigcup_{n\geq k} f^{-n}(L_r(n)).
$$
Since $\mu$ is invariant, we have
$\mu(f^{-n}(L_r(n)))=\mu(L_r(n))$ for all $n$. Then
$$
\sum_{n=1}^{\infty} \mu(L_r(n))
 = \sum_{n=1}^\infty\sum_{n_1\ge r \gamma(n)}\mu(H_{n_1})
 = \sum_{m=1}^\infty\sum_{n=1}^{[\gamma^{-1}(m/r)]}\mu(H_{m})
 \leq \sum_{m=1}^\infty \gamma^{-1}(m/r) \mu(H_{m}).
$$
Thus, using the hypothesis that $\gamma^{-1}(n_1(\cdot)/r)$ is integrable,
$$
\sum_{n=1}^{\infty} \mu(L_r(n))
 \leq \sum_{m=1}^\infty  \gamma^{-1}(m/r) \mu(H_{m})
 = \int \gamma^{-1}(n_1(x)/r) d\mu(x) < \infty.
$$
By the Borel-Cantelli lemma, this implies that $L$ has measure
zero. It follows that $\mu(D)=\mu(L)=0$, as claimed.

\end{proof}

\begin{corollary}\label{c.nonlacunar} If $\mu$ is a invariant expanding ergodic measure, then the sequence of hyperbolic times is nonlacunary.
\end{corollary}

\begin{proof} 
Observe that $n_1$ is integrable by Lemma~\ref{lem:01}. To finish the proof, just put $\gamma(t)=t$ in the Lemma~\ref{l.nonlacunargeneral}.
\end{proof}

\begin{lemma}

For $\mu$-a.e. $x \in M$ we have that given $\epsilon > 0$ small enough, if
${B}_\epsilon(n,x)$ is a dynamical ball at $x$ for $f$, then
there exists a periodic point in ${B}_n(x,\epsilon)$, with
period less than $n + K(n,x,\epsilon)$ where
\[
\limsup_{n \to \infty} \frac{K(n,x, \epsilon)}{n} = 0.
\]
\end{lemma}
\begin{proof}

Let us consider as before $g = f^\ell$ with $\ell$ chosen in such
a way that $\mu$-almost all $x\in M$ has a nonlacunary infinite
sequence $(n_k)_k$ of $g$-hyperbolic times (Lemma \ref{l.hyptimes} and
Corollary~\ref{c.nonlacunar}). We can assume that $x$ is such a point on
$\supp(\mu)$,  the support of $\mu$. Let $\epsilon > 0$ and $n
\geq 1$ be fixed with $\epsilon<\delta$ and let $k$ be such that $\ell n_k < n \leq \ell n_{k+1}$. 

From the uniform continuity of $f$ it follows that
there exists $\gamma=\gamma(\epsilon)<\epsilon$ such that for all $y\in M$
and $0\leq k\leq \ell$,
\begin{equation}\label{eq.gamma}
f^k\big(B(y,\gamma)\big)\subset B(f^k(y),\epsilon).
\end{equation}

As a consequence we have
\begin{equation} \label{eq.bolasdinamicas}
B_{g,n_{k+1}}(x,\gamma) \subset B_{f,n}(x,\epsilon).
\end{equation}

Let
$V(x)=g^{-n_{k+1}}\big(B(g^{n_{k+1}}(x),\gamma)\big)$
be the $g$-hyperbolic pre-ball around $x$  of length $n_{k+1}$.
Since
$n_{k+1}$ is a $g$-hyperbolic time for
$x$ and $g^{-j}$ is a contraction on
$B(g^{n_{k+1}}(x),\gamma)$
for all $1\leq j\leq n_{k+1}$,
it follows that $g^{n_{k+1}}(V(x))= B(g^{n_{k+1}}(x),\gamma)$ and that
$V(x) = B_{g,n_{k+1}}(x,\gamma)$.

As $f$ is strongly transitive on the support of $\mu$, we can find
$r> 0$ and $N(\epsilon) \in \NN$ such that for all $y \in \supp(\mu)$ we have
$B(y,r) \subset f^j(B(x,\gamma))$ for some $j \leq N(\epsilon)$.

Without loss of generality, we may assume that $n$ satisfies
$e^{-4cn}\gamma < r$ which gives immediately $V(x)=B_{g,n_{k+1}}(x,\gamma)  \subset B(x,r)$. Thus,  we conclude that
\[
V(x) \subset f^j(B(f^{\ell n_{k+1}}(x),\gamma)) = f^j\big(f^{\ell n_{k+1}}(V(x))\big)
\] for some $j \leq N$. Put $K(n,x,\epsilon) =  \ell n_{k+1} + j - n$. By Brower's Fixed Point Theorem, we have that

\[
f^{n+K(n,x,\epsilon)}: V(x) \rightarrow
f^{n+K(n,x,\epsilon)}(V(x)),
\]
has a fixed point. Using the nonlacunarity of $(n_k)_k$ we get that
\[
 \frac{K(n,x,\epsilon)}{n} \leq
 \frac{\ell n_{k+1} - n + N(\epsilon)}{n} \leq
 \frac{\ell n_{k+1} - \ell n_k + N(\epsilon)}{n_k} \to 0 \;\;\;\text{ when } k \to \infty.
\] This finish the proof of this lemma.

\end{proof}

%
 
 \section{Proof of Theorem~\ref{th:t1}}

In this section we prove Theorem~\ref{th:t1}. We need to prove that given an ergodic
expanding measure $\mu$,  $\epsilon>0$,  $\eta>0$ small enough and
$q$ a $\eta$-slowly varying function, for $\mu$-almost every point
$x \in M$ we have that the nonuniform dynamical ball
$\widetilde{B}_n(x,\epsilon,q)$   has a periodic point with period
less than $n + K(n,x,\epsilon,\eta)$, where
\[
\lim_{\eta\to 0} \limsup_{n \to \infty}
\frac{K(n,x,\epsilon,\eta)}{n} = 0.
\]

\begin{proof}[Proof of Theorem~\ref{th:t1}]

By Lemma \ref{l.hyptimes} and Corollary~\ref{c.nonlacunar}, there exists $\ell$ such that
almost every $x \in \supp(\mu)$ has infinitely many hyperbolic
times $n_1<n_2<\dots$ for $g=f^\ell$ and that $(n_k)$ is nonlacunary sequence.
Assume that $\eta<c$ and $\epsilon < \delta$. Denote by $\widetilde{B}_n(x,\epsilon,q)$ the nonuniform $(n,\epsilon, q)$-dynamical ball with respect to $f$ and by
$\widetilde{B}_n^\ell(x,\epsilon,q)$ the $ (n,\epsilon, q)$-dynamical
ball with respect to $f^\ell$. 

Observe that since  $\|Df\|$ uniformly is bounded from above and, by the Mean Value inequality, given $\epsilon>0$
we may choose $\alpha=\alpha(\epsilon)<\epsilon$ such that given $y\in M$
and $0\leq k\leq \ell$,
$$
f^k\big(B(y,r \alpha)\big)\subset B(f^k(y),r \alpha),
$$ for every $r>0$. As a consequence, given $m\in \mathbb{N}$, we have that
\begin{equation}\label{eq.boladinamica2}
\widetilde{B}_{\ell m}(x,\alpha(\epsilon),q) \subset \widetilde{B}_m^\ell(x,\epsilon,q).
\end{equation}

 Given $n\in \mathbb{N}$ big enough, take $\ell n_i\geq
n>\ell n_{i-1}$ two consecutive hyperbolic times for $x$. By the
definition of $\eta$-slowly varying function, we have that
$q(g^n(x))\leq e^{\ell n\eta}q(x)$. Observing that $g_x^{-j}$ is a
contraction for $1\leq j \leq n_i$, we have that
$$
 B(g^{n_i}(x), e^{-2\ell n_i \eta}q(x)^{-2}\epsilon)\subset
 g^{n_{i}}(\widetilde{B}_{n_i}^\ell (x,\epsilon,q)).
$$

Now, we choose $n_{i+s}$ a hyperbolic time for $x$ such that
\begin{equation}\label{eq.n}
n_{i+s-1}<\frac{c+\eta}{c}n_i \quad \text{  and } \quad
 n_{i+s}\geq \frac{c+\eta}{c}n_i.
\end{equation}

Observe that the choice of $n_{i+s}$ implies that
$e^{-c(n_{i+s}-n_i)}\epsilon \leq e^{-n_i\eta}\epsilon$. Since
$$
B_{n_{i+s}-n_i}^\ell(g^{n_i}(x), \epsilon) \subset
B(g^{n_i}(x), e^{-2c(n_{i+s}-n_i)}q(x)^{-2}\epsilon) \subset B(g^{n_i}(x)),e^{-2n_i
\eta} q(x)^{-2}\epsilon),
$$
we have that
$$
g^{n_{i+s}-n_i}\big(B_{n_{i+s}-n_i}^\ell (g^{n_i}(x))\big) =
B(g^{n_{i+s}}(x), q(x)^{-2}\epsilon) \subset
g^{n_{i+s}}\big(\widetilde{B}_{n_i}^\ell (x,\epsilon,q) \big).
$$

Choosing $\alpha$ as in Equation~(\ref{eq.boladinamica2}) and using that $\ell n_i>n$, we have that 
$$
B(f^{\ell n_{i+s}}(x), q(x)^{-2}\epsilon) \subset f^{\ell n_{i+s}}\big(\widetilde{B}_{n} (x,\alpha,q) \big).
$$

Since $f$ is strongly transitive on $\supp(\mu)$, we have that for
$n$ big enough and some $0\leq J \leq N(q(x)^{-2}\epsilon)$, the map $f^{\ell n_{i+s}+ J}$
restricted to $\widetilde{B}_{n}(x,\alpha,q)$ is onto
$\widetilde{B}_{n}(x,\alpha,q) $. Thus, by Brower's Fixed Point Theorem, $f$ has a periodic point 
on $\widetilde{B}_{n}(x,\alpha,q) \subset \widetilde{B}_{n}(x,\epsilon,q).$ Define
$$
K(n,x,\epsilon, \eta)=\ell n_{i+s}+J-n,
$$ then  $f$ has a periodic point with period at most $n+K(n,x,\epsilon,
  \eta)$.
To finish the proof, we need to check that
 $$
\lim_{\eta\to 0} \limsup_{n \to \infty}
\frac{K(n,x,\epsilon,\eta)}{n} = 0.
$$

In fact, since $(n_k)$ is a nonlacunary sequence, we have that
$n_{i}/n_{i-1}$ and $n_{i+s}/n_{i+s-1}$ converge to $1$ when $n$
goes to infinity (observe that $i$  goes to infinity when $n$ goes
to infinity). By the left-hand side of Equation~(\ref{eq.n}) above:
$$
\frac{n_{i+s}}{n_i} \leq \frac{n_{i+s}}{n_{i+s-1}}
\frac{c+\eta}{c}.
$$
Observing that since $n_{i-1} \leq n \leq n_i$ and
$n_i/n_{i-1}\rightarrow 1$ when $n$ goes to infinity, we have that
$n_i/n \rightarrow 1$ when $n$ goes to infinity. Using this and
the equation above, we conclude that
$$
\lim_{n\rightarrow \infty} \frac{n_{i+s}}{n} \leq
\frac{c+\eta}{c}.
$$
In particular,
$$
\lim_{\eta\to 0} \limsup_{n \to \infty}
\frac{K(n,x,\epsilon,\eta)}{n} = \lim_{\eta\to 0} \limsup_{n \to
\infty} \frac{n_{i+s}+J-n}{n} \leq \lim_{\eta \rightarrow 0}
\frac{c+\eta}{c} - 1= 0,
$$ since $0\leq J\leq N(q(x)^{-2}\epsilon)$, and so the Theorem~\ref{th:t1} is proved.
\end{proof}

\bigskip

\noindent {\bf Acknowledgements.} Part of this paper was written during a visit to IMPA and IBILCE/UNESP. The author gratefully acknowledges the hospitality of these institutions and financial support of CNPq, FAPERJ, FAPEAL and  FAPESP. The author is also grateful  to N. Muniz for several useful discussions and to F. Ledrappier, A. Katok, W. Sun and M. Viana for suggestions.


\def\polhk#1{\setbox0=\hbox{#1}{\ooalign{\hidewidth \lower1.5ex\hbox{`}\hidewidth\crcr\unhbox0}}}

\bibliographystyle{plain}
\bibliography{bib06072010}

\bigskip

\flushleft

{\bf Krerley Oliveira} (krerley\@@gmail.com)\\
Instituto de Matem\'{a}tica, UFAL \\
57072-090 Maceió, AL, Brazil.

\bigskip

\flushleft
%

\end{document}